\documentclass{amsart}

\usepackage{amssymb}

\usepackage[dvips]{graphicx}

\usepackage{color}

\theoremstyle{plain}

\newtheorem{mainthm}{Theorem}

\newtheorem*{conj*}{Conjecture}
\newtheorem*{cor*}{Corollary}
\newtheorem*{def*}{Definition}

\newtheorem{theorem}{Theorem}[section]
\newtheorem{proposition}[theorem]{Proposition}
\newtheorem{corollary}[theorem]{Corollary}
\newtheorem{lemma}[theorem]{Lemma}
\newtheorem{conj}[theorem]{Conjecture}

\theoremstyle{definition}

\newtheorem{definition}[theorem]{Definition}

\newtheorem{remark}[theorem]{Remark}
\newtheorem{q}{Question}

\newcommand{\Z}{\mathbb{Z}}
\newcommand{\N}{\mathbb{N}}

\newcommand{\eps}{\varepsilon}

\title{N-expansive homeomorphisms with the shadowing property}

\author[Bernardo Carvalho and Welington Cordeiro]{}
\thanks{2010 \emph{Mathematics Subject Classification}: Primary 37B99; Secondary 37D99.}
 \keywords{Expansive, $n$-expansive, shadowing, transitive, limit shadowing.}

 \email{bernardocarvalho@ufv.br}
 \email{wud11@psu.edu}

\begin{document}

\maketitle
\centerline{\scshape Bernardo Carvalho, Welington Cordeiro}
{\footnotesize
} 
\begin{abstract}{We discuss the dynamics of $n$-expansive homeomorphisms with the shadowing property defined on compact metric spaces. For every $n\in\N$, we exhibit an $n$-expansive homeomorphism, which is not $(n-1)$-expansive, has the shadowing property and admits an infinite number of chain-recurrent classes. We discuss some properties of the local stable (unstable) sets of $n$-expansive homeomorphisms with the shadowing property and use them to prove that some types of the limit shadowing property are present. This deals some direction to the problem of non-existence of topologically mixing $n$-expansive homeomorphisms that are not expansive.}
\end{abstract}



\bigskip


\section{Introduction and Statement of Results}\label{sec:intro}

The dynamics of expansive homeomorphisms with the shadowing property may be very complicated but it is quite well understood (see Aoki and Hiraide's monograph \cite{AH}). It is known that these systems admit only a finite number of chain recurrent classes and that other pseudo-orbit tracing properties are present, such as the limit shadowing property (see Section \ref{2} for definitions). A generalization of the expansiveness property that has been given attention recently is the $n$-expansive property (see \cite{Art}, \cite{APV}, \cite{LZ}, \cite{Mor}).

\begin{definition}
We say that a homeomorphism $f$, defined in a metric space $(X,d)$, is \emph{n-expansive} $(n\in\N)$ if there exists $c>0$ such that for every $x\in X$ the set $$\Gamma(x,c):=\{y\in X \,\,\, ; \,\,\, d(f^n(x),f^n(y))\leq c, \,\,\, n\in\Z\}$$ contains at most $n$ different points of $X$. The number $c$ is called the \emph{$n$-expansivity constant} of $f$.
\end{definition}

The expansive homeomorphisms are exactly the $1$-expansive ones. Roughly speaking, the $n$-expansive homeomorphisms may admit $n$ different orbits `moving together' but cannot admit $n+1$. In this paper, we analyze the dynamics of $n$-expansive homeomorphisms with the shadowing property defined on compact metric spaces. The first result that does not hold in this scenario is the Spectral Decomposition Theorem (Theorem 3.1.11 in \cite{AH}).

\begin{mainthm}\label{Example}
For every $n\in\N$, there is an $n$-expansive homeomorphism, defined in a compact metric space, that is not $(n-1)$-expansive, has the shadowing property and admits an infinite number of chain recurrent classes.
\end{mainthm}

This example also enlightens an important difference between the expansive homeomorphisms and the $n$-expansive ones: for expansive homeomorphisms there is some number $\eps>0$ such that for every $x\in X$, the local stable set of $x$ of size $\eps$ ($W^s_{\eps}(x)$) is contained in the stable set of $x$ ($W^s(x)$) (see \cite{Ma} for a proof), while for $n$-expansive homeomorphisms such number does not exist (even when the shadowing property is present). However, we are able to prove some similar property: the existence of such number in a uniform way along the orbits.

\begin{mainthm}\label{A} If an $n$-expansive homeomorphism $f$, defined in a compact metric space $X$, has the shadowing property, then for each $x\in X$ there exists $\eps_x>0$ such that $$W^s_{\eps_x}(f^m(x))\subset W^s(f^m(x))$$ for every $m\in\mathbb{Z}$.
\end{mainthm}

Using this Theorem we are able to prove that some types of limit shadowing are still present (precise definitions are given in Section \ref{2}).

\begin{mainthm}\label{B}
If an $n$-expansive homeomorphism has the shadowing property then it has the limit shadowing property. If, in addition, it is topologically mixing, then it has the two-sided limit shadowing property.
\end{mainthm}



These results generalize Theorems C and D in \cite{CK} to the $n$-expansive scenario and gives some direction to the following question:

\begin{q}\label{q}
Does there exist a topologically mixing $n$-expansive homeomorphism that is not expansive?
\end{q}

The two-sided limit shadowing property is known to be one of the strongest notions of pseudo-orbit tracing properties since it implies many of them (it implies even the specification property, see \cite{CK}). Its importance relies on its relation with hyperbolicity and transitivity, as one can see in \cite{C} and \cite{C2}. Thought there are examples of non-expansive homeomorphisms with the two-sided limit shadowing property \cite{CK}, the only known examples of homeomorphisms defined on compact and connected finite dimensional manifolds satisfying it are topologically conjugated to a transitive Anosov diffeomorphism, being, in particular, expansive. It is expected that, in this scenario, these are the only ones (see \cite{Cthesis}). If this happens to be true, then Theorem \ref{B} would answer Question \ref{q} negatively for homeomorphisms defined on compact and connected finite dimensional manifolds and admitting the shadowing property.

The paper is organized as follows: in Section \ref{2} we state some necessary definitions, in Section \ref{W} we prove Theorem \ref{Example}, in Section \ref{P} we prove Theorem \ref{A} and in Section \ref{5} we prove Theorem \ref{B}.

\section{Definitions}\label{2}

In this section we state all definitions that will be necessary in the proofs of our results. During this section $f$ denotes a homeomorphism defined in a compact metric space $(X,d)$. The first definition is the (standard) shadowing property.

\begin{definition}
We say that a sequence $(x_k)_{k\in\Z}\subset X$ is a \emph{$\delta$-pseudo-orbit} if
it satisfies $$d(f(x_k),x_{k+1})<\delta, \,\,\,\,\,\,k\in\Z.$$
A sequence $(x_k)_{k\in\Z}\subset X$ is $\eps$-\emph{shadowed} if there exists $y\in X$ satisfying $$d(f^k(y),x_k)<\eps, \,\,\,\,\,\, k\in\Z.$$ We say that $f$ has the \emph{shadowing property} if for every $\eps>0$ there exists $\delta>0$ such that every two-sided $\delta$-pseudo-orbit is
$\eps$-shadowed.
\end{definition}

This property was extensively studied due to its relation to the hyperbolic and stability thoeries (see \cite{AH}, \cite{P}). Now we define the limit shadowing property.

\begin{definition}
We say that $(x_k)_{k\in\N}$ is a \emph{limit pseudo-orbit} if it
satisfies $$d(f(x_k),x_{k+1})\rightarrow 0, \,\,\,\,\,\, k\rightarrow\infty.$$
A sequence $\{x_k\}_{k\in\N_0}$ is \emph{limit-shadowed} if there exists $y\in X$ such that $$d(f^k(y),x_k)\rightarrow 0, \,\,\,\,\,\, k\rightarrow
\infty.$$ We say that $f$ has the \emph{limit shadowing property} if every limit pseudo-orbit is limit-shadowed.
\end{definition}

This property was introduced by T. Eirola, O. Nevanlinna, S. Pilyugin in \cite{ENP} and further studied by S. Pilyugin in \cite{P1}. The first author of this paper proved (Lemma 2.1 in \cite{C}) that any expansive homeomorphism, defined in a compact metric space, with the shadowing property has the limit shadowing property. This gives many examples of homeomorphisms with the limit shadowing property. If we consider bilateral sequences of $X$ we can define a property called two-sided limit shadowing.

\begin{definition}
We say that a sequence $(x_k)_{k\in\Z}$ of points of a metric space $(X,d)$ is a \emph{two-sided limit pseudo-orbit} if it satisfies $$d(f(x_k),x_{k+1})\rightarrow 0, \,\,\,\,\,\, |k|\rightarrow\infty.$$
A sequence $(x_k)_{k\in\Z}\subset X$ is \emph{two-sided limit shadowed} if there exists $y\in X$ satisfying $$d(f^k(y),x_k)\rightarrow 0, \,\,\,\,\,\,
|k|\rightarrow \infty.$$ In this situation, we also say that $y$ \emph{two-sided limit shadows} $(x_k)_{k\in\Z}$. We say that
$f$ has the \emph{two-sided limit shadowing property} if every two-sided limit pseudo-orbit is two-sided limit shadowed.
\end{definition}

Tough this property is very similar to define, it is very different from the limit shadowing property. The first author studies this property in detail in \cite{Cthesis}, \cite{C}, \cite{C2} and \cite{CK}. Among many results, it is proved that this property implies the shadowing property, the limit shadowing property, the average shadowing property, the assymptotic average shadowing property and even the specification property, being one of the strongest known notions of pseudo-orbit tracing properties. Now the specification property is defined.

\begin{definition}
Let $\tau=\{I_1,\dots,I_m\}$ be a finite collection of disjoint finite subsets of consecutive integers, $I_i=[a_i,b_i]\cap\Z$ for some $a_i,b_i\in\Z$, with
$$
a_1\le b_1 < a_2\le b_2 <\ldots < a_m\le b_m.$$
Let a map $P\colon \bigcup_{i=1}^mI_i\rightarrow X$ be such that for each $I\in \tau$ and $t_1, t_2\in I$ we have $$f^{t_2-t_1}(P(t_1))=P(t_2).$$
We call a pair $(\tau,P)$ a \emph{specification}. We say that the specification $S=(\tau,P)$ is \emph{$L$-spaced} if $a_{i+1}\geq b_i+L$ for all $i\in\{1,\dots,m-1\}$. Moreover, $S$ is \emph{$\eps$-shadowed} by $y\in X$ for $f$ if $$d(f^k(y),P(k))<\eps \,\,\,\,\,\, \textrm{for every} \,\,\, k\in \bigcup_{i=1}^mI_i.$$
We say that a homeomorphism $f\colon X\rightarrow X$ has the \emph{specification property} if for every $\eps>0$ there exists $L\in\N$ such that every $L$-spaced specification is $\eps$-shadowed.
\end{definition}



Every continuous map with the specification property is topologically mixing.

\begin{definition}
We say that $f$ is \emph{transitive} if for any pair $(U,V)$ of non-empty open subsets of $X$ there is $k\in\N$ such that $f^k(U)\cap V\neq\emptyset$. We say that $f$ is \emph{totally transitive} if all its iterates $f^k$, $k\in\N$, are transitive. We say that $f$ is \emph{topologically mixing} if for any pair $(U,V)$ of non-empty open subsets of $X$ there is $k\in\N$ such that $f^j(U)\cap V\neq\emptyset$ for every $j\geq k$.
\end{definition}

It is known that the topologically mixing property implies the totally transitive property, which, in turn, implies the transitive property. In the non-transitive scenario, the notion of a \emph{chain-recurrent class} is important.

\begin{definition}
The \emph{chain-recurrent class} of a point $x\in X$ is the set of all points $y\in X$ such that for every $\eps>0$ there exist a finite $\eps$-pseudo orbit starting at $x$ and ending at $y$ and another finite $\eps$-pseudo orbit starting at $y$ and ending at $x$.
\end{definition}

It is easy to see that transitive homeomorphisms admit only one chain recurrent class, while the example we exhibit in Theorem \ref{Example} admits an infinite number of such classes. A generalization of the notion of $n$-expansivity is the \emph{finite expansivity} and we define it as follows.

\begin{definition}
We say that $f$ is \emph{finite expansive} if there exists $c>0$ such that for every $x\in X$ the set $$\Gamma(x,c):=\{y\in X \,\,\, ; \,\,\, d(f^n(x),f^n(y))\leq c, \,\,\, n\in\Z\}$$ is finite. The number $c$ is called the \emph{finite-expansivity constant} of $f$.
\end{definition}

There are finite expansive homeomorphisms that are not $n$-expansive for any $n\in\N$ (see Remark \ref{example}) even when the shadowing property is present. We finish this section with the definitions of the stable and unstable sets of $f$ (also the local ones).

\begin{definition}
The \emph{stable set} of $x\in X$ is the set
$$W^s(x)=\{y\in X; \,\,\, d(f^k(y),f^k(x))\rightarrow0, \,\,\, \textrm{if} \,\,\, k\rightarrow\infty\}.$$
The \emph{unstable set} of $x\in X$ is the set
$$W^u(x)=\{y\in X; \,\,\, d(f^{-k}(y),f^{-k}(x))\rightarrow0, \,\,\, \textrm{if} \,\,\, k\rightarrow\infty\}.$$
For some number $\eps>0$, we define the $\eps$-\emph{stable set} of $x\in X$ as the set
$$W_{\eps}^s(x)=\{y\in X; \,\,\, d(f^k(y),f^k(x))\leq\eps, \,\,\, k\in\N\},$$
and define the $\eps$-\emph{unstable set} of $x\in X$ as the set
$$W_{\eps}^u(x)=\{y\in X; \,\,\, d(f^{-k}(y),f^{-k}(x))\leq\eps, \,\,\, k\in\N\}.$$
\end{definition}

\section{Proof of Theorem \ref{Example}}\label{W}

Consider an expansive homeomorphism $g$ defined in a compact metric space $(M,d_0)$ and satisfying the shadowing property. Further, suppose it has an infinite number of periodic points $\{p_k\}_{k\in\N}$, which we can suppose belong to different orbits. Define $X$ as the set $M\cup E$ where $E$ is an infinite enumerable set.

For each $i\in\{1,\dots,n-1\}$, each $k\in\N$ and each $j\in\{0,\dots,\pi(p_k)-1\}$ consider a point $q(i,k,j)\in E$. We can choose these points so that every point of $E$ appears once and only once in $\{q(i,k,j)\}_{i,k,j}$. Define a distance $d$ on $X$ by
$$d(x,y)=\begin{cases}d_0(x,y), &x,y\in M,\\
\frac{1}{k}+d_0(y,g^j(p_k)), &x=q(i,k,j),\,\,\, y\in M,\\
\frac{1}{k}, &x=q(i,k,j),\,\,\, y=q(l,k,j),\,\,\, i\neq l,\\
\frac{1}{k}+\frac{1}{m}+d_0(g^j(p_k),g^r(p_m)), &x=q(i,k,j),\,\,\, y=q(l,m,r),\\
&\hspace{+1.0cm} k\neq m \,\,\, \text{or}\,\,\, j\neq r.
\end{cases}$$
To check that $d$ is a compact metric on $X$ is an exercise that we leave to the reader. Define the homeomorphism $f:X\rightarrow X$ by $$f(x)=\begin{cases} g(x), &x\in M,\\
q(i,k,(j+1)\hspace{-0.3cm}\mod\pi(p_k)), &x=q(i,k,j).
\end{cases}$$ Note that $E$ splits into an infinite number of periodic orbits of $f$. Indeed, for each $i\in\{1,\dots,n-1\}$ and $k\in\N$, the set $\{q(i,k,j)\,;\,j\in\{0,\dots,\pi(p_k)-1\}\}$ is a periodic orbit for $f$ with period $\pi(p_k)$. Now we will check the announced properties of $f$:

\begin{enumerate}
\item $f$ is $n$-expansive\vspace{+0.3cm}

Since $g:M\to M$ is expansive it admits an expansivity constant $\delta>0$. Suppose that $n+1$ different points of $X$ belong to the same dynamic ball of radius $\delta$. The expansiveness of $f_{|M}=g$ assures that at most one of these points belong to $M$ and, consequently, at least $n$ of them belong to $E$. Moreover, at least two of these points are of the form $q(i,k,j)$ and $q(l,m,r)$ with $k\neq m$. Indeed, if this is not the case then two of them are of the form $q(i,k,j)$ and $q(i,k,r)$ with $j\neq r$. It follows that for each $s\in\Z$ we have
\begin{eqnarray*}d(g^s(g^j(p_k)),g^s(g^r(p_k)))&=&
d(f^s(q(i,k,j)),f^s(q(i,k,r)))-\frac{2}{k}\\
&\leq & d(f^s(q(i,k,j)),f^s(q(i,k,r)))\\
&\leq & \delta.
\end{eqnarray*}
This implies that $g^j(p_k)=g^r(p_k)$, which, in turn, implies that $j=r$ and we obtain a contradiction. Now note that for every $s\in\Z$ the following holds:
\begin{eqnarray*}
\hspace{+0.95cm}d(g^s(g^j(p_k)),g^s(g^r(p_m)))&=&
d(f^s(q(i,k,j)),f^s(q(l,m,r)))-\frac{1}{k}-\frac{1}{m}\\
&\leq &d(f^s(q(i,k,j)),f^s(q(l,m,r)))\\
&\leq &\delta.
\end{eqnarray*}
Therefore, $g^j(p_k)=g^s(p_m)$ and $p_m=p_k$, which is a contradiction with the fact that $k\neq m$. It is important to note that this can be done since $M$ has an infinite number of periodic orbits $\{p_k\}_{k\in\N}$.

\vspace{+0.3cm}
\item $f$ is not $(n-1)$-expansive\vspace{+0.3cm}

For each $\delta>0$ choose $k\in\N$ such that $\frac{1}{k}<\delta$ and note that the dynamic ball $\Gamma(p_k,\frac{1}{k})$ contains $n$ different points since for each $i\in\{1,\dots,n-1\}$ the point $q(i,k,0)$ belongs to $\Gamma(p_k,\frac{1}{k})$. This implies that $\Gamma(p_k,\delta)$ contains at least $n$ different points and that $f$ is not $(n-1)$-expansive, since this can be done for any $\delta>0$.

\vspace{+0.3cm}
\item $f$ has the shadowing property\vspace{+0.3cm}

Since $g$ has the shadowing property, for each $\epsilon>0$ we can consider $\delta_g>0$ such that every $\delta_g$-pseudo orbit of $g$ is $\frac{\epsilon}{2}$-shadowed. Choose $m\in\mathbb{N}$ such that $$\frac{1}{m}<\min\left\{\frac{\epsilon}{2},\frac{\delta_g}{3}\right\}$$ and let $\delta=\frac{1}{m}$. If $\{x_i\}_{i\in\mathbb{Z}}\subset X$ is a $\delta$-pseudo orbit of $f$ then either $\{x_i\}_{i\in\mathbb{Z}}$ is one of the orbits $\{q(l,k,j)\,;\,j\in\{0,\dots,\pi(p_k)-1\}\}$, with $l\in\{1,\dots,n-1\}$ and $k\in\{1,\dots,m-1\}$, or $\{x_i\}_{i\in\mathbb{Z}}$ does not contain any point of these orbits. In the first case $\{x_i\}_{i\in\mathbb{Z}}$ is obviously shadowed, so we will focus on the second case. Thus if $x_i=q(l,k,j)$ then $k\geq m$.

Define a sequence $\{y_i\}_{i\in\mathbb{Z}}\subset M$ by
$$y_i=\begin{cases}x_i,&x_i\in M,\\
g^j(p_k),&x_i=q(l,k,j).
\end{cases}$$
The sequence $(y_i)_{i\in\Z}$ is a $\delta_g$-pseudo orbit for $g$ since for each $i\in\Z$ the following holds:
\begin{eqnarray*}\hspace{+1.2cm}d(g(y_i),y_{i+1})&=&
d(f(y_i),y_{i+1})\\
&\leq& d(f(y_i),f(x_i))+d(f(x_i),x_{i+1})+d(x_{i+1},y_{i+1}) \\
&\leq&\frac{1}{m}+\frac{1}{m}+\frac{1}{m}\\
&\leq&\delta_g.
\end{eqnarray*}
Then there exists $x\in M$ such that $$d(g^i(x),y_i)<\frac{\eps}{2}, \,\,\,\,\,\, i\in\Z.$$
It follows that $(x_i)_{i\in\Z}$ is $\eps$-shadowed by $x$ since for each $i\in\Z$ the following holds:
\begin{eqnarray*}
d(f^i(x),x_i)&\leq&d(f^i(x),y_i)+d(y_i,x_i)\\
&\leq&\frac{\eps}{2}+\frac{1}{m}\\
&\leq&\eps.
\end{eqnarray*}
Since this can be done for any $\eps>0$ we obtain that $f$ has the shadowing property.

\vspace{+0.3cm}
\item $f$ admits an infinite number of chain-recurrent classes\vspace{+0.3cm}

Different periodic orbits in $E$ belong to different chain-recurrent classes. Indeed, every point $q(i,k,j)\in E$ satisfies $$d(q(i,k,j),x)\geq\frac{1}{k}, \,\,\,\,\,\, x\in X\setminus\{q(i,k,j)\}.$$ This means that if $0<\eps<\frac{1}{k}$ then the orbit of $q(i,k,j)$ cannot be connected by $\eps$-pseudo orbits with any other point of $X$. This proves that the chain recurrent class of $q(i,k,j)$ contains only its orbit and we conclude the proof.
\end{enumerate}

\begin{remark}
This construction can be done starting with any expansive homeomorphism with the shadowing property that admits an infinite number of periodic points. Examples of these systems are the Anosov diffeomorphisms and good references of these systems are (\cite{A}, \cite{AH}, \cite{S} and many others).
\end{remark}

\begin{remark}\label{example}
This example may be slightly modificated to obtain a finite expansive homeomorphism that is not $n$-expansive for any $n\in\N$, has the shadowing property and admits an infinite number of chain recurrent classes. Instead of adding $n-1$ periodic orbits one just have to add $k-1$ periodic orbits near each $p_k$. In this case, the set $\Gamma(p_k,\frac{1}{k})$ will contain $k$ different points of $X$ instead of $n$. The details are left to the reader.
\end{remark}

\begin{remark}
This theorem was born from some examples discussed by the second author of this paper and A. Artigue during a research visit to the Universidad de la Rep\'ublica in Uruguay, though these examples were thought in another setting.
\end{remark}

\section{Local stable sets}\label{P}

The example of Theorem \ref{Example} also shows that there is no number $\eps>0$ such that $W_{\eps}^s(x)\subset W^s(x)$ and $W_{\eps}^u(x)\subset W^u(x)$ for every $x\in X$. Indeed, for every $\eps>0$ one can consider $k\in\N$ such that $\frac{1}{k}<\eps$ and note that all the periodic points $q(i,k,j)$, with $i\in\{1,\dots,n-1\}$, belong to $W_{\eps}^s(p_k)$ and $W_{\eps}^u(p_k)$ but do not belong to $W^s(p_k)$ nor $W^u(p_k)$. However, in this section, we prove Theorem \ref{A}, which gives us a similar property. The first step is to prove the following proposition.

\begin{proposition}\label{local2} If an $n$-expansive homeomorphism $f$, defined in a compact metric space $X$, has the shadowing property, then for each $x\in X$ there exists $\eps_x>0$ such that $W^s_{\eps_x}(x)\subset W^s(x)$ and $W^u_{\eps_x}(x)\subset W^u(x)$.
\end{proposition}

One can easily check that the conclusion of this proposition holds for the example of Theorem \ref{Example}. Toward proving this proposition we discuss the \emph{number of different stable $($unstable$)$ sets} of a homeomorphism in a local stable (unstable) set.

\begin{definition}\label{number}
For some number $\eps>0$ and some point $x\in X$ we say that a positive integer number $n(x,\eps)$ is the \emph{number of different stable sets} of $f$ in $W^s_{\eps}(x)$ if
\begin{enumerate}
  \item there exists a set $E(x,\eps)\subset W_{\eps}^s(x)$ with $n(x,\eps)$ elements such that if two different points $y,z\in E(x,\eps)$ then $y\notin W^s(z)$,
  \item if $y_1,y_2,\dots,y_{n(x,\eps)+1}$ are $n(x,\eps)+1$ different points of $W_{\eps}^s(x)$, then there exist two different points $y_i,y_j\in\{y_1,y_2,\dots,y_{n(x,\eps)+1}\}$ such that $y_i\in W^s(y_j)$.
\end{enumerate}
We define the \emph{number of different unstable sets} of $f$ in $W^u_{\eps}(x)$ in a similar way and denote it by $\bar{n}(x,\eps)$.
\end{definition}

It is obvious that $n(x,\eps)\geq1$ and $\bar{n}(x,\eps)\geq1$ for every $x\in X$ and $\eps>0$. It is also easy to see that $n(x,\eps)=1$ if and only if $W_{\eps}^s(x)\subset W^s(x)$, and that $\bar{n}(x,\eps)=1$ if and only if $W^u_{\eps}(x)\subset W^u(x)$. So for an expansive homeomorphism defined in a compact metric space $X$ it is known that there exists $\eps>0$ such that $n(x,\eps)=\bar{n}(x,\eps)=1$ for every $x\in X$. We are interested in studying the numbers $n(x,\eps)$ and $\bar{n}(x,\eps)$ for $n$-expansive homeomorphisms. We conjecture the following is true:

\begin{conj}
If $f$ is an $n$-expansive homeomorphism defined in a compact metric space $X$, then there exists $\eps>0$ such that $$n(x,\bar{\eps})\leq n$$ for every $x\in X$ and $0<\bar{\eps}\leq\eps$.
\end{conj}

We prove this conjecture when the shadowing property is present.

\begin{proposition}\label{local} If an $n$-expansive homeomorphism $f$, defined in a compact metric space $X$, has the shadowing property, then there exists $\eps>0$ such that $$n(x,\bar{\eps})\leq n$$ for every $x\in X$ and $0<\bar{\eps}\leq\eps$.
\end{proposition}

\begin{proof} Let $c>0$ be the $n$-expansivity constant of $f$, $\eps=\frac{c}{4}$ and $0<\bar{\eps}\leq\eps$. First, note that for any $x\in X$, the inequality $n(x,\bar{\eps})\leq n$ is equivalent to the following property: if $y_1,y_2,\dots, y_{n+1}$ are $n+1$ different points of $W_{\bar{\eps}}^s(x)$ then there exist two different points $y_i,y_j\in\{y_1,y_2,\dots,y_{n(x,\bar{\eps})+1}\}$ such that $y_i\in W^s(y_j)$.

Then, suppose that there are $n+1$ different points $y_1,y_2,...,y_{n+1}\in W^s_{\bar{\eps}}(x)$ satisfying $y_i\notin W^s(y_j)$ whenever $i\neq j$. This implies that there exists a number $r>0$ (which we can suppose smaller than $\bar{\eps}$) such that for each pair $(i,j)\in\{1,2,\dots,n+1\}\times\{1,2,\dots,n+1\}$ satisfying $i<j$, there exists a sequence of positive integer numbers $(a^{(i,j)}_m)_{m\in\N}$ satisfying $$a^{(i,j)}_m\to\infty, \,\,\,\,\,\, m\to\infty,$$ and such that $$d(f^{a^{(i,j)}_m}(y_i),f^{a^{(i,j)}_m}(y_j))>r, \,\,\,\,\,\, m\in\N.$$

In what follows we split the proof in two cases:

\vspace{+0.5cm}

\textbf{Case 1:} There exist $r'>0$ and a sequence of positive integer numbers $(a_m)_{m\in\N}$ such that $$a_m\rightarrow\infty, \,\,\,\,\,\, m\to\infty$$ and for each pair $(i,j)\in\{1,2,\dots,n+1\}\times\{1,2,\dots,n+1\}$ the following holds $$d(f^{a_m}(y_i),f^{a_m}(y_j))>r', \,\,\,\,\,\, m\in\N.$$

\vspace{+0.5cm}

The compactness of $X$ assures the existence of a subsequence $(a_{m_k})_{k\in\N}$ of $(a_m)_{m\in\N}$ and points $x_0,x_1,\dots,x_{n+1}\in X$ such that $f^{a_{m_k}}(x)\rightarrow x_0$ and $f^{a_{m_k}}(y_i)\rightarrow x_i$ for each $i\in\{1,2,\dots,n+1\}$. Moreover, for each $i\in\{1,2,\dots,n+1\}$ and each $m\in\Z$ the following holds:
\begin{eqnarray*}
d(f^m(x_i),f^m(x_0))&=&\lim_{k\rightarrow\infty} d(f^m(f^{a_{m_k}}(y_i)),f^m(f^{a_{m_k}}(x)))\\
&\leq&\bar{\eps}.
\end{eqnarray*}
The last inequality holds since $y_i\in W^s_{\bar{\eps}}(x)$ for each $i\in\{1,2,\dots,n+1\}$ and $a_{m_k}\to\infty$ when $k\to\infty$. Therefore, $x_i\in\Gamma(x_0,\bar{\eps})\subset\Gamma(x_0,c)$ for each $i\in\{1,2,\dots,n+1\}$. Since $c$ is an $n$-expansivity constant, the set $\Gamma(x_0,c)$ contains at most $n$ different points of $X$.

But the set $\{x_1,x_2,\dots,x_{n+1}\}$ is contained in $\Gamma(x_0,c)$ and $x_i\neq x_j$ whenever $i\neq j$, since\begin{eqnarray*}
d(x_i,x_j)&=&\lim_{k\rightarrow\infty}d(f^{a_{m_k}}(y_i),f^{a_{m_k}}(y_j))\\
&\geq& r'\\
&>&0.
\end{eqnarray*}
This is a contradiction.

\vspace{+0.5cm}

\textbf{Case 2:} The assumption of the first case is not satisfied.

\vspace{+0.5cm}

In this case, we can find (at least) two different indices $i,j\in\{1,...,n,n+1\}$ and a sequence of positive integer numbers $(b_m)_{m\in\N}$ satisfying $$b_m\rightarrow\infty, \,\,\,\,\,\, m\to\infty$$ and such that $$d(f^{b_m}(y_i),f^{b_m}(y_j))<\frac{1}{m}, \,\,\,\,\,\, m\in\N.$$ We can choose the sequence $(b_m)_{m\in\N}$ so that $$b_m<a^{(i,j)}_m<b_{m+1}, \,\,\,\,\,\, m\in\N.$$ Let $\alpha=\min\{\bar{\eps},\frac{r}{4}\}$ and consider the number $\beta>0$ (given by the shadowing property) such that every $\beta$-pseudo orbit is $\alpha$-shadowed. Choose $w\in\mathbb{N}$ such that $$\frac{1}{w}<\beta.$$ For each $l\in\{1,2,\dots,n+1\}$ we define a sequence $(x_k^l)_{k\in\Z}$ as follows:

$$x_k^l=\begin{cases}f^k(y_i), & k<b_{w+l} \,\,\, \text{or} \,\,\, k\geq b_{w+l+1},\\
f^k(y_j), & b_{w+l}\leq k<b_{w+l+1}.
\end{cases}$$

These sequences are $\beta$-pseudo orbits for $f$, because for $k=b_{w+l}-1$ we have

\begin{eqnarray*}d(f(x^l_{b_{w+l}-1}),x^l_{b_{w+l}})&=& d(f(f^{b_{w+l}-1}(y_i)),f^{b_{w+l}}(y_j)) \\
&=& d(f^{b_{w+l}}(y_i),f^{b_{w+l}}(y_j)) \\
&<&\frac{1}{w+l}\\
&<&\frac{1}{w}\\
&<&\beta,
\end{eqnarray*}

for $k=b_{w+l+1}-1$ we have

\begin{eqnarray*}d(f(x^l_{b_{w+l+1}-1}),x^l_{b_{w+l+1}}) &=&d(f(f^{b_{w+l+1}-1}(y_j)),f^{b_{w+l+1}}(y_i)) \\
&=& d(f^{b_{w+l+1}}(y_j),f^{b_{w+l+1}}(y_i)) \\
&<&\frac{1}{w+l+1}\\
&<&\frac{1}{w}\\
&<&\beta
\end{eqnarray*}
and $d(f(x^l_k),x_{k+1}^l)=0$, otherwise. The shadowing property assures the existence of $n+1$ points $z_1,z_2,\dots,z_{n+1}\in X$ such that $(x^l_k)_{k\in\Z}$ is $\alpha$-shadowed by $z_l$ for each $l\in\{1,2,\dots,n+1\}$. These points are different because if $t$ and $s$ are different indices in $\{1,2,...,n+1\}$ and $p=a_{w+t}^{(i,j)}$ then

\begin{eqnarray*}d(f^p(z_t),f^p(z_s))&\geq& d(x^t_p,x^s_p)-d(f^p(z_t),x^t_p)-d(f^p(z_s),x^s_p) \\
&=& d(f^p(y_j),f^p(y_i))-d(f^p(z_t)),x^t_p)-d(f^p(z_s),x^s_p)\\
&\geq& r-\alpha-\alpha\\
&\geq&\frac{r}{2}\\
&>&0.
\end{eqnarray*}

Moreover, the set $\{z_1,z_2,\dots,z_{n+1}\}$ is contained in $\Gamma(z_1,c)$. Indeed, if $k\leq0$ then $x^l_k=f^k(y_i)$ for every $l\in\{1,2,\dots,n+1\}$ and this implies that \begin{eqnarray*}d(f^k(z_1),f^k(z_l))&\leq& d(f^k(z_1),f^k(y_i))+d(f^k(y_i),f^k(z_l))\\
&\leq& d(f^k(z_1),x^1_k)+d(x^l_k,f^k(z_l))\\
&\leq& \alpha+\alpha\\
&\leq& 2\bar{\eps}\\
&\leq& c.
\end{eqnarray*}
If $k>0$ then \begin{eqnarray*}d(x^1_k,x^l_k)&\leq& d(x^1_k,x)+d(x,x^l_k)\\
&\leq& 2\bar{\eps}
\end{eqnarray*}

which implies that

\begin{eqnarray*} d(f^k(z_1),f^k(z_l))&\leq& d(f^k(z_1),x^1_k)+d(x^1_k,x^l_k)+d(x^l_k,f^k(z_l)) \\
&\leq& \alpha+2\bar{\eps}+\alpha\\
&\leq& 4\bar{\eps}\\
&\leq& c.
\end{eqnarray*}

This proves that $\Gamma(z_1,c)$ contains $n+1$ different points of $X$ and contradicts the fact that $c$ is an $n$-expansivity constant of $f$. This finishes the proof.
\end{proof}

\begin{remark}\label{finitenatural}
We note that the argument of the previous lemma can be slightly modified to prove that if a finite expansive homeomorphism has the shadowing property, then there exists $\eps>0$ such that $n(x,\bar{\eps})\in\N$ for every $x\in X$ and $0<\bar{\eps}\leq\eps$. We leave this as an exercise, though.
\end{remark}

Now we are ready to prove Proposition \ref{local2}. \vspace{+0.4cm}

\hspace{-0.45cm}\emph{Proof of Proposition \ref{local2}} : Let $\eps>0$ (given by Proposition \ref{local}) be such that $n(x,\eps)\leq n$ for every $x\in X$. Definition \ref{number} assures the existence of a set $E(x,\eps)$ with $n(x,\eps)$ points of $W^s_{\eps}(x)$ such that any two different points in $E(x,\eps)$ belong to different stable sets. This set has the additional property that any point $y\in W^s_\eps(x)$ belong to the stable set of some point $z\in E(x,\eps)$.

We can assume, without loss of generality, that $x\in E(x,\eps)$. Thus, if $z\in E(x,\eps)\setminus\{x\}$ then there exist $r_z>0$ and a sequence $(a_m^z)_{m\in\N}$ of positive integer numbers satisfying $$a_m^z\rightarrow\infty, \,\,\,\,\,\, m\to\infty$$ and such that $$d(f^{a^z_m}(z),f^{a^z_m}(x))>r_z, \,\,\,\,\,\, m\in\N.$$ If we consider the number $$\eps_x=\frac{1}{4}\min\{r_z\,;\,z\in E(x,\eps)\setminus\{x\}\}$$ then any point $y\in W_{\eps_x}^s(x)$ must belong to $W^s(x)$ because for any $z\in E(x,\eps)\setminus\{x\}$ and any $m\in\N$ we have
\begin{eqnarray*}d(f^{a^z_m}(z),f^{a^z_m}(y))&\geq& d(f^{a^z_m}(z),f^{a^z_m}(x))-d(f^{a^z_m}(x),f^{a^z_m}(y)) \\
&\geq& r_z-\eps_x\\
&\geq& \frac{r_z}{2},
\end{eqnarray*}
which implies that $y\notin W^s(z)$. This proves that $W^s_{\eps_x}(x)\subset W^s(x)$. The proof for the local unstable set is similar and we leave the details to the reader. \qed

\vspace{+0.5cm}

Now we prove some additional properties of the number $n(x,\eps)$.

\begin{lemma}\label{finite}
If $f$ is a finite expansive homeomorphism, defined in a compact metric space $X$, then there exists $\eps>0$ such that $$n(x,\bar{\eps})\leq n(f(x),\bar{\eps})$$ for every $x\in X$ and $0<\bar{\eps}\leq\eps$.
\end{lemma}

\begin{proof}
Since $f$ is finite expansive there exists $\eps>0$ such that for every $x\in X$ the number $n(x,\eps)$ is well defined (see Remark \ref{finitenatural}). For each $x\in X$ and $0<\bar{\eps}\leq\eps$ we consider the sets $E(x,\bar{\eps})$ and $E(f(x),\bar{\eps})$ given by Definition \ref{number}. Note that $n(x,\bar{\eps})$ is exactly the cardinality of the set $E(x,\bar{\eps})$ and that $n(f(x),\bar{\eps})$ is the cardinality of the set $E(f(x),\bar{\eps})$.

We define a map $r:E(x,\bar{\eps})\to E(f(x),\bar{\eps})$ as follows: for any $z\in E(x,\bar{\eps})$ we define $r(z)$ as the point in $E(f(x),\bar{\eps})$ such that $f(z)\in W^s(r(z))$. Note that this map is well defined since $z\in W^s_{\bar{\eps}}(x)$ implies that $f(z)\in W^s_{\bar{\eps}}(f(x))$ and then condition (2) in Definition \ref{number} assures the existence of the point $r(z)$. The map $r$ is injective, since $r(z)=r(w)$ implies that $f(z)\in W^s(f(w))$, which, in turn, implies that $z\in W^s(w)$ and then condition (1) in Definition \ref{number} implies that $z=w$. The map $r$ being injective implies the desired inequality.
\end{proof}

An easy corollary of the previous lemma is the following:

\begin{corollary}
If $f$ is a finite expansive homeomorphism, defined in a compact metric space $X$, then there exists $\eps>0$ such that $$n(x,\bar{\eps})\leq n(f^k(x),\bar{\eps})$$ for every $x\in X$, $0<\bar{\eps}\leq\eps$ and $k\in\N$.
\end{corollary}

In the Example of Remark \ref{example} the number of different stable sets is bounded along the orbits. The following question is still unanswered:

\begin{q}
Does there exists a finite expansive homeomorphism, defined in a compact metric space, such that the sequence $(n(f^k(x),\eps))_{k\in\N}$ converges to infinity for some $x\in X$?
\end{q}

\begin{remark}
We note that such example (if it exists) cannot be $n$-expansive for any $n\in\N$ due to Lemma \ref{local}.
\end{remark}

In the $n$-expansive scenario an additional property holds.

\begin{lemma}\label{l}
If $f$ is an $n$-expansive homeomorphism, defined in a compact metric space $X$, then there exist $\eps>0$ such that for every $x\in X$ there is $l(x)\in\N$ satisfying $$n(f^{l(x)}(x),\bar{\eps})=n(f^{l(x)+k}(x),\bar{\eps})$$ for every $0<\bar{\eps}<\eps$ and $k\in\N$.
\end{lemma}

\begin{proof}
Let $\eps>0$ (given by Proposition \ref{local}) be such that $n(x,\eps)\leq n$ for every $x\in X$. If $x\in X$ and $0<\bar{\eps}<\eps$ then there exists $l(x)\in\N$ such that $$n(f^{l(x)}(x),\bar{\eps})\geq n(f^{l(x)+k}(x),\bar{\eps})$$ for every $k\geq 0$. Otherwise, there will be a sequence $(a_k)_{k\in\N}$ of positive integer numbers such that $$n(f^{a_k-1}(x),\bar{\eps})<n(f^{a_k}(x),\bar{\eps})$$ for every $k\in\N$. This implies that the sequence $(n(f^{a_k}(x),\bar{\eps}))_{k\in\N}$ converges to infinity as $k$ goes to infinity, but this contradicts Proposition \ref{local}. The desired equality follows from Lemma \ref{finite}.
\end{proof}



We are finally ready to prove Theorem \ref{A}.

\vspace{+0.3cm}

\hspace{-0.45cm}\emph{Proof of Theorem \ref{A}} : Let $\eps>0$ (given by Lemma \ref{l}) be such that for every $x\in X$ there is $l(x)\in\N$ satisfying $$n(f^{l(x)}(x),\bar{\eps})=n(f^{l(x)+k}(x),\bar{\eps})$$ for every $0<\bar{\eps}<\eps$ and $k\in\N$. For any $x\in X$ let $\eps_x>0$ (given by Proposition \ref{local2}) be such that $$W^s_{\eps_x}(f^{l(x)}(x))\subset W^s(f^{l(x)}(x)).$$ This implies that $$n(f^{l(x)}(x),\eps_x)=1$$ and hence $$n(f^{l(x)+k}(x),\eps_x)=1, \,\,\,\,\,\, k\in\N.$$ Lemma \ref{finite} proves that $n(f^m(x),\eps_x)=1$ for every $m\in\Z$ and this finishes the proof. \qed

\begin{remark}
One can also prove that for every $x$ there exists $\eps_x>0$ such that $$W^u_{\eps_x}(f^m(x))\subset W^u(f^m(x)), \,\,\,\,\,\, m\in\Z.$$
\end{remark}

The following questions seems natural.

\begin{q}
Is the shadowing property necessary in Theorem \ref{A}?
\end{q}

\section{Proof of Theorem \ref{B}}\label{5}

The first statement of Theorem \ref{B} will be proved in a separate proposition since it generalizes Lemma 2.1 in \cite{C}.

\begin{proposition}\label{ShadImpLimit} If an $n$-expansive homeomorphism $f$, defined in a compact metric space $X$, has the shadowing property, then $f$ and $f^{-1}$ have the limit shadowing property.
\end{proposition}

\begin{proof} We must prove that any limit pseudo orbit $(x_k)_{k\in\mathbb{N}}$ of $f$ is limit shadowed. Let $\eps>0$ (given by Proposition \ref{local}) be such that $$n(x,\bar{\eps})\leq n$$ for every $x\in X$ and $0<\bar{\eps}\leq\eps$. For each $j\in\N$, the shadowing property assures the existence of a number $\delta_j>0$ such that every $\delta_j$-pseudo orbit is $\frac{\eps}{j+1}$-shadowed. Choose an increasing sequence $(k_j)_{j\in\mathbb{N}}$ of positive integer numbers such that $$d(f(x_k),x_{k+1})<\delta_j, \,\,\,\,\,\, k\geq k_j.$$ For each $j>1$ we define a sequence $(x^j_k)_{k\in\mathbb{Z}}$ by $$x_k^j=\begin{cases}x_{k+k_j}, & k>0 \\
f^k(x_{k_j}), & k\leq0.
\end{cases}$$
This sequence is a $\delta_j$-pseudo orbit of $f$, so the shadowing property assures the existence of $z_j\in X$ such that $$d(f^k(z_j),x^j_k)<\frac{\eps}{j+1}, \,\,\,\,\,\, k\in\Z.$$ For each $j\in\N$, let $$y_j=f^{-k_j}(z_j).$$ Let $l(y_1)\in\N$ (given by Lemma \ref{l}) be such that $$n(f^{l(y_1)}(y_1),\bar{\eps})=n(f^{l(y_1)+k}(y_1),\bar{\eps})$$ for every $0<\bar{\eps}<\eps$ and $k\in\N$. We will prove that one point of the set $$f^{-l(y_1)}(E(f^{l(y_1)}(y_1),\eps))$$ limit shadows $(x_k)_{k\in\N}$. This set has exactly $n(f^l(y_1),\eps)$ elements, so we can write $$f^{-l(y_1)}(E(f^{l(y_1)}(y_1),\eps))=\{p_i\in X\,\,;\,\,i\in\{1,\dots,n(f^l(y_1),\eps)\}.$$ Note that if $j\in\N$ and $k\geq k_j$ then
\begin{eqnarray*}
d(f^k(y_j),f^k(y_1))&\leq&d(f^k(y_j),x_k)+d(x_k,f^k(y_1))\\
&<&\frac{\eps}{j+1}+\frac{\eps}{2}\\
&<&\eps.
\end{eqnarray*}
This implies that $$f^{k_j}(y_j)\in W^s_{\eps}(f^{k_j}(y_1)), \,\,\,\,\,\, j\in\N.$$ So, for each $j\in\N$, there is $$w_j\in E(f^{k_j}(y_1),\eps)$$ such that $$f^{k_j}(y_j)\in W^s(w_j).$$ We can suppose $k_2\geq l(y_1)$, so that for each $j\geq2$ there is $$i\in\{1,\dots,n(f^l(y_1),\eps)\}$$ such that $$w_j=f^{k_j}(p_i).$$ Hence, there is $i\in\{1,\dots,n(f^l(y_1),\eps)\}$ and a subsequence $(k_{j_m})_{m\in\N}$ of $(k_j)_{j\in\N}$ such that $$w_m=f^{k_{j_m}}(p_i), \,\,\,\,\,\, m\in\N.$$ We claim that $p_i$ limit shadows $(x_k)_{k\in\N}$. Indeed, since $$f^{k_{j_m}}(y_{j_m})\in W^s(f^{k_{j_m}}(p_i)), \,\,\,\,\,\, m\in\N,$$ it is easily seen that $$y_{j_m}\in W^s(p_i), \,\,\,\,\,\, m\in\N.$$ For each $\alpha>0$, choose $m\in\N$ satisfying $$\frac{\eps}{m+1}<\frac{\alpha}{2}.$$ Thus, if $k\geq k_{j_m}$ then $$d(f^k(y_{j_m}),x_k)<\frac{\alpha}{2}.$$ Since $$y_{j_m}\in W^s(p_i)$$ we can choose $K\geq k_{j_m}$ such that $$d(f^k(y_{j_m}),f^k(p_i))\leq\frac{\alpha}{2}, \,\,\,\,\,\, k\geq K.$$ Then for every $k\geq K$ we have
\begin{eqnarray*}
d(f^k(p_i),x_k)&\leq&d(f^k(p_i),f^k(y_{j_m}))+d(f^k(y_{j_m}),x_k)\\
&<&\frac{\alpha}{2}+\frac{\alpha}{2}\\
&<&\alpha.
\end{eqnarray*}
Since this can be done for any $\alpha>0$ we proved that $p_i$ limit shadows $(x_k)_{k\in\N}$. This proves that $f$ has the limit shadowing property. To prove that $f^{-1}$ also has it one just have to note that $f^{-1}$ is $n$-expansive and has the shadowing property and apply above argument.
\end{proof}

\hspace{-0.45cm}\emph{Proof of Theorem \ref{B}} : 
We will follow the proof of the expansive case in \cite{C} exchanging the number $\eps>0$, given by the expansiveness property, by Theorem \ref{A}. We assume that $f$ is a topologically mixing $n$-expansive homeomorphism with the shadowing property and we prove that $f$ has the two-sided limit shadowing property. Let $\{x_k\}_{k\in\mathbb{Z}}$ be a two-sided limit pseudo orbit of $f$. Since $f$ and $f^{-1}$ have the limit shadowing property (Proposition \ref{ShadImpLimit}) there exist $p_1,p_2\in X$ satisfying $$d(f^k(p_1),x_k)\to0, \,\,\,\,\,\, k\to-\infty,$$ and $$d(f^k(p_2),x_k)\to0, \,\,\,\,\,\, k\to\infty.$$ Let $\eps_1>0$ and $\eps_2>0$ (given by Theorem \ref{A}) be such that $$W^u_{\eps_1}(f^k(p_1))\subset W^u(f^k(p_1)),$$ and $$W^s_{\eps_2}(f^k(p_2))\subset W^s(f^k(p_2))$$ for every $k\in\mathbb{Z}$. Let $$\eps=\min\{\eps_1,\eps_2\}$$ and choose $\delta>0$ (given by the shadowing property) such that every $\delta$-pseudo orbit of $f$ is $\eps$-shadowed. Since $f$ is topologically mixing and has the shadowing property, it has the specification property. Hence, there is $M\in\mathbb{N}$ such that every $M$-spaced specification is $\delta$-shadowed. Choose $N\in\N$ such that $2N\geq M$ and that for every $k\geq N$ the following holds $$d(f^{-k}(p_1),x_{-k})<\delta \,\,\,\,\,\, \textrm{and} \,\,\,\,\,\, d(f^k(p_2),x_k)<\delta.$$

Let $I_1=\{-N\}$, $I_2=\{N\}$, $P(-N)=f^{-N}(p_1)$ and $P(N)=f^N(p_2)$. Since $(\{I_1,I_2\};P)$ is a $M$-spaced specification, there is $z\in X$ satisfying
$$d(f^{-N}(z),f^{-N}(p_1))=d(f^{-N}(z),P(-N))<\delta$$ and $$d(f^N(z),f^N(p_2))=d(f^N(z),P(N))<\delta.$$ This implies that the sequence $(y_k)_{k\in\Z}$
defined by
$$y_k=\begin{cases}
f^k(p_1), & k<-N\\
f^k(z), &-N\leq k\leq N\\
f^k(p_2), &k> N
\end{cases}
$$
is a $\delta$-pseudo orbit of $f$. Then there is $\tilde{z}\in X$ that $\eps$-shadows it. In particular, $$d(f^k(\tilde{z}),f^k(p_1))<\eps,
\,\,\,\,\,\, k\leq-N$$ and $$d(f^k(\tilde{z}),f^k(p_2))<\eps, \,\,\,\,\,\, k\geq N.$$ This implies that $$f^{-N}(\tilde{z})\in W^u_{\eps}(f^{-N}(p_1))\subset W^u(f^{-N}(p_1))$$ and that $$f^N(\tilde{z})\in W^s_{\eps}(p_2)\subset W^s(f^N(p_2)).$$ Thus we obtain
$$d(f^k(\tilde{z}),f^k(p_1))\rightarrow0, \,\,\,\,\,\, k\rightarrow-\infty$$ and $$d(f^k(\tilde{z}),f^k(p_2))\rightarrow0, \,\,\,\,\,\,
k\rightarrow\infty.$$ Since $p_1$ limit-shadows in the past $(x_k)_{k\in-\N_0}$ and $p_2$ limit-shadows $(x_k)_{k\in\N_0}$ it follows that $\tilde{z}$
two-sided limit shadows $(x_k)_{k\in\Z}$. This finishes the proof. \qed

\ \\

\end{document}